\documentclass[final,3p]{elsarticle}

\usepackage[utf8]{inputenc}
\usepackage{amsmath}
\usepackage{amsfonts}
\usepackage{amssymb}
\usepackage{amsfonts}
\usepackage{amsthm}
\usepackage{amscd}
\usepackage{color}
\usepackage{graphics}
\usepackage{graphicx}
\usepackage{overpic}
\usepackage[normalem]{ulem}

\usepackage[all]{xy}

\usepackage{algorithm,algorithmicx,algpseudocode}

\algrenewcommand\algorithmicrequire{\textbf{INPUT:}}
\algrenewcommand\algorithmicensure{\textbf{OUTPUT:}}

\usepackage{natbib}
\biboptions{sort&compress,numbers}

\usepackage[T1]{fontenc}

\usepackage{lineno}

\definecolor{orange}{rgb}{1,0.5,0}

\newtheorem{theorem}{\bf Theorem}
\newtheorem{lemma}[theorem]{\bf Lemma}
\newtheorem{proposition}[theorem]{\bf Proposition}

\newtheorem{remark}[theorem]{\bf Remark}

\newtheorem{example}[theorem]{\bf Example}
\newtheorem{corollary}[theorem]{\bf Corollary}

\def\bfb{\boldsymbol{b}}

\def\bfx{\boldsymbol{x}}
\def\bfy{\boldsymbol{y}}

\def\bfu{\boldsymbol{u}}
\def\bfv{\boldsymbol{v}}

\def\para{\vspace{1.5mm}}

\begin{document}

\sloppy

\begin{frontmatter}
\title{Nonzero Constant Wronskians of Polynomials and Laurent Polynomials, and Geometric Consequences}

\author[a]{Carlos Hermoso}
\ead{juange.alcazar@uah.es}
\author[a]{Juan Gerardo Alc\'azar}
\ead{juange.alcazar@uah.es}

%\fntext[proy]{Supported by the grant PID2020-113192GB-I00 (Mathematical Visualization: Foundations, Algorithms and %Applications) from the Spanish MICINN.}
%\fntext[proy2]{Member of the Research Group {\sc asynacs} (Ref. {\sc ccee2011/r34}) }
%\fntext[proy3]{Supported by the National Natural Science Foundation of China under Grant 61872332, 11731013}

%\fntext[proy4]{Partially supported by a Giner de los R\'{\i}os Grant
%  of the Universidad de Alcal\'a, ANR JCJC GALOP (ANR-17-CE40-0009), a
%  public grant as part of the Investissement d'avenir project
%  reference ANR-11-LABX-0056-LMH LabEx LMH (PGMO ALMA), the PHC
%  GRAPE, and from the 2232 International Fellowship for Outstanding Researchers Program from T\"UBITAK (Project No: 118C240). }

\address[a]{Departamento de F\'{\i}sica y Matem\'aticas, Universidad de Alcal\'a,
E-28871 Madrid, Spain}
\begin{abstract}
We characterize the polynomials $p_1(t),\ldots,p_n(t)$ whose Wronskian $W(p_1,\ldots,p_n)$ is a nonzero constant. Then, we generalize our results to characterize the Laurent polynomials with the same property. Finally, for rational functions we prove an impossibility result for $n=2$, and pose the case $n\geq 3$ as an open question, although we suggest an impossibility conjecture. Some geometric consequences are derived, especially in the case of polynomials. 
\end{abstract}
\end{frontmatter}

%\linenumbers

\section{Introduction} 

The Wronskian $W(f_1,\ldots,f_n)$ of $n$ functions is a classical tool related to the linear dependence of functions, commonly appearing when studying systems of Differential Equations: as it is well-known, if the functions are linearly dependent then $W(f_1,\ldots,f_n)$ is identically zero. Furthermore, under some general conditions fulfilled by polynomials, rational functions and in fact analytic functions, the converse is also true \cite{BO}. 

\para
While the case when $W(f_1,\ldots,f_n)=0$ has been well studied, in this paper we ask ourselves whether, and under what conditions, it can happen that $W(f_1,\ldots,f_n)$ is a nonzero constant. We investigate this question for polynomials and Laurent polynomials, i.e. polynomial expressions admitting finitely many terms with negative exponents; we also provide some observations for the case of rational functions. 

\para
This might seem like an odd question to ask, so let us motivate our interest a bit. Given functions $f_1(t),\ldots,f_n(t)$, the expression $\bfy(t)=(f_1(t),\ldots,f_n(t))$ is a parametrization of a curve in dimension $n$. Then the Wronskian of $f_1,\ldots,f_n$ can be written as (we denote by $|\bullet|$ the determinant of $\bullet$)
\[
W(f_1,\ldots,f_n)=W(\bfy(t))=|\bfy(t), \bfy'(t),\ldots,\bfy^{(n-1)}(t)|,
\]
i.e. the determinant whose rows consist of $\bfy(t)$ and the derivatives of $\bfy(t)$ up to $n-1$-th order. This determinant may not sound familiar, but if we replace $\bfy(t):=\bfx'(t)$ then we have 
\begin{equation}\label{motiv}
|\bfx'(t),\bfx''(t),\ldots,\bfx^{(n)}(t)|,
\end{equation}
which appears in Differential Geometry as the numerator of the formula for the curvature, in the case of planar curves, and of the torsion, for space curves; it also appears when one considers curvatures of curves in $n$-space for arbitrary $n$ \cite{S21}. If this expression is constant, then, for instance, we cannot have (real or complex) values of the parameter $t$ where the curvature, for planar curves, or the torsion, for space curves, is zero. 

\para
But, continuing with our motivation, the actual context where we ran into the problem treated in this paper was the question of algorithmically detecting (Euclidean) symmetries of polynomial and rational curves, and the detection of affine and projective equivalences between curves of this kind. For isometries, the expression in Eq. \eqref{motiv}, or more generally its square, is a \emph{differential invariant}, because it does not change when an isometry $f(\bfx)=A\bfx+\bfb$ is applied. Indeed, 
\[
|A\bfx'(t),A\bfx''(t),\ldots,A\bfx^{(n)}|=|A||\bfx'(t),\bfx''(t),\ldots,\bfx^{(n)}(t)|,
\]
so if $|A|=\pm 1$, the square of Eq. \eqref{motiv} stays invariant. This is not true for general affine or projective mappings, but in those cases one can easily generate other invariants from Eq. \eqref{motiv}; for instance, for affine equivalences one can replace one of the derivatives in Eq. \eqref{motiv} by a derivative $\bfx^{(k)}(t)$ with $k\geq n+1$, and then take the quotient with Eq. \eqref{motiv}. 

\para
These ideas, involving Eq. \eqref{motiv} and related invariants, have been used, for instance, in \cite{AHM15} and, more recently, in \cite{AGH24} (see the bibliographies of those papers for a more complete list of references) for algorithmically computing symmetries, affine and projective equivalences of polynomial and rational curves. However, when doing so, the case when Eq. \eqref{motiv} is a constant leads to special situations: in those cases, the above invariants are useless and one needs to, first, identify the geometric conditions when this happens, and then provide ad-hoc solutions for those conditions. It was in this context that we came up with the problem considered here.

\para
In this paper we first investigate the problem for polynomials, and characterize the case when its Wronskian is a nonzero constant. Then we extend our results to Laurent polynomials, where we also provide a characterization; furthermore, the result generalizes naturally to rational functions with just one pole. Finally, for general rational functions, i.e. for rational functions with more than one pole, we prove an impossibility result for $n=2$ and conjecture that this is also the case for $n\geq 3$. Additionally, and mainly for polynomials, we derive some geometric consequences and a connection with the \emph{rational normal curve}, well known in Algebraic Geometry \cite{CC07,C86,V82}. 

\para
The structure of this paper is the following. We address the case of polynomials in Section \ref{back}. The case of Laurent polynomials is treated in Section \ref{laurent}. General rational functions are studied in Section \ref{sec-impossible}. We close with our conclusion and some ideas for further work in Section \ref{sec-conclusion}.

\section*{Acknowledgments}

Juan Gerardo Alc\'azar and Carlos Hermoso are supported by the grant PID2020-113192GB-I00 (Mathematical Visualization: Foundations, Algorithms and Applications) from the Spanish MICINN. Juan G. Alc\'azar and Carlos Hermoso are also members of the Research Group {\sc asynacs} (Ref. {\sc ccee2011/r34}). We also acknowledge here the ideas of the paper \cite{BO}, which in turn inspired some of the techniques and ideas of this paper. 

\section{Constant nonzero Wronskians of polynomials}\label{back}

Let ${\Bbb K}$ be a field of characteristic zero, and let ${\Bbb K}[t]$ be the set of polynomials in the variable $t$ with coefficients in ${\Bbb K}$. Furthermore, for $i=1,\ldots,n$, let $p_i(t)\in {\Bbb K}[t]$,
\[
p_i(t)=a_{i,d_i}t^{d_i}+a_{i,d_i-1}t^{d_i-1}+\cdots +a_{i,0}
\]
where $a_{i,d_i}\neq 0$, so $p_i(t)$ is a polynomial in $t$ of degree $d_i$. The \emph{Wronskian} of $p_1(t),\ldots,p_n(t)$ is the determinant
\begin{equation}\label{wronskian}
W(p_1,\ldots,p_n)=\left\vert \begin{array}{cccc} p_1(t) & p_2(t) & \cdots & p_n(t)\\
p'_1(t) & p'_2(t) & \cdots & p'_n(t)\\
\vdots & \vdots & \ddots & \vdots\\
p_1^{(n-1)}(t) & p_2^{(n-1)}(t) & \cdots & p_n^{(n-1)}(t)
\end{array}
\right\vert,
\end{equation}
with $p_j^{(k)}(t)$ the $k$-th derivative of $p_j(t)$. It is well-known that if $W(p_1,\ldots,p_n)\neq 0$, then $p_1(t),\ldots,p_n(t)$ are linearly independent. In the case of polynomials (and in fact, for a wider class of functions), the converse is also true \cite{BO}. If $W(p_1,\ldots,p_n)$ is a constant, and if the value of the constant is zero, then we can be sure that the $p_i(t)$ are linearly dependent. The question that we ask ourselves is: \emph{what must happen for $W(p_1,\ldots,p_n)$ to be a nonzero constant?} This can certainly happen: for instance, if $p_1(t)=t^2+t$, $p_2(t)=2t^2$, $p_3(t)=t-2$ then
\[
W(p_1,p_2,p_3)=\left \vert \begin{array}{ccc}
t^2+t & 2t^2 & t-2\\
2t+1 & 4t& 1\\
2 & 4 & 0
\end{array} \right\vert = -8.
\]

\para
So let us explore this question. In order to do this, let ${\bf d}=\mbox{max}\{d_1,\ldots,d_n\}$. Then ${\Bbb K}_{\bf d}[t]$, the set of polynomials of degree $\leq {\bf d}$, is a vector space of dimension ${\bf L}={\bf d}+1$. Let us denote by ${\mathcal V}=\langle p_1(t),\ldots,p_n(t)\rangle$ the subspace of ${\Bbb K}_{\bf d}[t]$ spanned by $p_1(t),\ldots,p_n(t)$. Since the $p_i(t)$ are linearly independent, $\mbox{dim}({\mathcal V})=n$; and since ${\mathcal V}\subset {\Bbb K}_{\bf d}[t]$, $n\leq {\bf d}+1$, i.e. ${\bf d}\geq n-1$, which is also a necessary condition for the Wronskian to be a nonzero constant. A first technical result is the following. 

\begin{proposition}\label{prop1}
Let $p_1(t),\ldots,p_n(t)\in {\Bbb K}[t]$, linearly independent, and let ${\mathcal V}=\langle p_1(t),\ldots,p_n(t)\rangle$. Then there exist $q_1(t),\ldots,q_n(t)\in {\Bbb K}[t]$, satisfying that: (i) the maximum of the degrees of the $p_i(t)$ and of the $q_i(t)$ is the same; (ii) the degrees of the $q_i(t)$ are all distinct; (iii) $W(p_1,\ldots,p_n)=W(q_1,\ldots,q_n)$; (iv) $\langle p_1(t),\ldots,p_n(t)\rangle=\langle q_1(t),\ldots,q_n(t)\rangle$. 
\end{proposition}

\begin{proof}
Let ${\bf d}=\mbox{max}\{d_1,\ldots,d_n\}$, where the degree of $p_i(t)$ is $d_i$. Without loss of generality, assume that $d_i\geq d_j$ for $i\leq j$ (otherwise we just rearrange the $p_i(t)$), so $d_1={\bf d}$. Let $p_1(t),\ldots,p_k(t)$, $k\leq n$, be the polynomials whose degree is equal to ${\bf d}$, and let ${\bf c}_j$ be the $j$-th column, $j\leq k$, of $W(p_1,\ldots,p_n)$, i.e. the column corresponding to $p_j(t)$ and its derivatives. Since the derivative is a linear operation, performing in $W(p_1,\ldots,p_n)$ the substitution 
\[
{\bf c}_j\to {\bf c}_j-\dfrac{a_{j,{\bf d}}}{a_{1,{\bf d}}}{\bf c}_1
\]
is equivalent to replacing $p_j(t)$ by $\tilde{p}_j(t)$, with
\begin{equation}\label{linop}
\tilde{p}_j(t)=p_j(t)-\dfrac{a_{j,{\bf d}}}{a_{1,{\bf d}}}p_1(t).
\end{equation}
In particular, we have that
\[
W(p_1,\ldots,p_n)=W(p_1,\ldots,p_{j-1},\tilde{p}_j,p_{j+1},\ldots,p_n).
\]
After performing this Gauss-like operation for $j=2,\ldots,k$, we reach a list of polynomials $p_1(t),\tilde{p}_2(t),\ldots,\tilde{p}_k(t),p_{k+1}(t),\ldots,p_n(t)$ where we just have one polynomial of degree ${\bf d}$, and whose Wronskian coincides with $W(p_1,\ldots,p_n)$. Continuing this process with the highest degree $\leq {\bf d}-1$, etc. we eventually reach a list of polynomials $q_1(t)=p_1(t),q_2(t)=\tilde{p}_2(t),\ldots,q_n(t)$, with $W(p_1,\ldots,p_n)=W(q_1,\ldots,q_n)$. Two of these polynomials cannot have the same degree $m>0$, because in that case we just continue the process. And two of these polynomials cannot be constant, either: indeed, none of these polynomials can be the zero polynomial since otherwise the Wronskian is zero, contrary to the assumption that the $p_i(t)$ are linearly independent; additionally, if two $q_i$ are constant, then the corresponding columns in $W(q_1,\ldots,q_n)$ are proportional, so $W(q_1,\ldots,q_n)=0$, which again contradicts the linear independence of the $p_i(t)$. Finally, notice that the $q_i(t)$ have been constructed by means of linear combinations of the $p_i(t)$, and therefore $\langle p_1(t),\ldots,p_n(t)\rangle=\langle q_1(t),\ldots,q_n(t)\rangle$. 
\end{proof}

Therefore, if $p_1(t),\ldots,p_n(t)$ have a nonzero constant Wronskian, then we can find $q_1(t),\ldots,q_n(t)$, with distinct degrees, whose Wronskian has the same value. The following step is to make precise the values of the degrees of the $q_i(t)$. In order to do this we first need the following proposition, related to the degree of the Wronskian: notice that since the entries of the Wronskian are polynomials in the variable $t$, the Wronskian itself must also be a polynomial in $t$ (possibly constant). 

\begin{proposition}\label{propdeg}
Let $q_1(t),\ldots,q_n(t)\in {\Bbb K}[t]$, and for $i=1,\ldots,n$ let $d_i$ be the degree of $q_i(t)$. Assume that all the $d_i$ are distinct. Then the degree of $W(q_1,\ldots,q_n)$ is 
\begin{equation}\label{deg}
\tau=d_1+\cdots+d_n-{n\choose 2}.
\end{equation}
\end{proposition}

\begin{proof}
Writing $q_i(t)=b_{i,d_i}t^{d_i}+\cdots$, $b_{i,d_i}\neq 0$, and expanding the determinant $W(q_1,\ldots,q_n)$ as a sum of determinants, the term of highest degree in $t$ of $W(q_1,\ldots,q_n)$ corresponds to $W(r_1,\ldots,r_n)$, where $r_i(t)=b_{i,d_i}t^{d_i}$. By Lemma 1 of \cite{BO}, we get that 
\[
W(r_1,\ldots,r_n)=V(d_1,\ldots,d_n)t^\tau \prod_{i=1}^n b_{i,d_i},
\]
where $V(d_1,\ldots,d_n)=\prod_{1\leq i<j\leq n}(d_j-d_i)$, i.e. the Vandermonde determinant whose $j$-th column is $[1,d_j,\ldots,d_j^{n-1}]^T$. Since the $d_i$ are all distinct, $V(d_1,\ldots,d_n)\neq 0$, and the result follows. 
\end{proof}

\begin{corollary}\label{cor1}
Let $q_1(t),\ldots,q_n(t)\in {\Bbb K}[t]$, and for $i=1,\ldots,n$ let $d_i$ be the degree of $q_i(t)$. Assume that $d_1>d_2>\ldots>d_n$ (so in particular all $d_i$ are distinct). If $W(q_1,\ldots,q_n)$ is a nonzero constant, then $d_1=n-1, d_2=n-2, \ldots,d_{n-1}=1, d_n=0$. 
\end{corollary}

\begin{proof}
Since the $d_i$ are all distinct, by Proposition \ref{propdeg} the degree $\tau$ in Eq. \eqref{deg} satisfies that 
\[
\tau=(d_1+\cdots +d_n)-{n\choose 2}.
\]
Since 
\[
{n \choose 2}=1+2+\cdots +(n-2)+(n-1),
\]
we can rearrange $\tau$ in the following way:
\begin{equation}\label{rearrange}
\tau=(d_1-(n-1))+(d_2-(n-2))+\cdots +(d_{n-1}-1)+d_n.
\end{equation}
Assume now that $d_1\geq n$. Then we have that $d_2\geq n-2$: indeed, if $d_2<n-2$ we cannot extract $n-1$ distinct values for $d_2,\ldots,d_{n-1,d_n}$ from the set $\{0,1,\ldots,n-3\}$, because $\{0,1,\ldots,n-3\}$ consists of $n-2$ distinct elements. Analogously we can see that $d_j\geq n-j$ for $j=3,\ldots,n-1$. Therefore, $d_1\geq n$ implies that $d_1-(n-1)>0$, and $d_j\geq n-j$ for $j=2,\ldots,n-1$ implies that $d_j-(n-j)\geq 0$ for $j=2,\ldots,n-1$; additionally, $d_n\geq 0$. Substituting in Eq. \eqref{rearrange} we get that $\tau>0$, so $W(q_1,\ldots,q_n)$ is not constant, which is a contradiction. Thus, we conclude that $d_1\leq n-1$. Since we have $n$ distinct $d_i$, we get that $d_1=n-1$, and once again because all the $d_i$ are distinct, the result follows. 
\end{proof}

Finally we can characterize the desired property. 

\begin{theorem}\label{th-main}
Let $p_1(t),\ldots,p_n(t)\in {\Bbb K}[t]$, linearly independent. Then $W(p_1,\ldots,p_n)$ is a nonzero constant if and only if we have that 
\begin{equation}\label{eq-th}
\begin{bmatrix}
p_1(t) \\
p_2(t)\\
\vdots \\
p_n(t)
\end{bmatrix}=A\cdot  
\begin{bmatrix}
1 \\
t\\
\vdots \\
t^{n-1}
\end{bmatrix}
\end{equation}
where $A$ is a nonsingular matrix. 
\end{theorem}

\begin{proof}
$(\Rightarrow)$ Assume that $W(p_1,\ldots,p_n)$ is a nonzero constant. By Proposition \ref{prop1}, there exist $q_1(t),\ldots,q_n(t)$ with distinct degrees, such that $W(p_1,\ldots,p_n)=W(q_1,\ldots,q_n)$ and such that $\langle p_1(t),\ldots,p_n(t)\rangle=\langle q_1(t),\ldots,q_n(t)\rangle$. This last condition implies that 
\begin{equation}\label{eq1}
\begin{bmatrix}
p_1(t) \\
p_2(t)\\
\vdots \\
p_n(t)
\end{bmatrix}=P_1\cdot  
\begin{bmatrix}
q_1(t) \\
q_2(t)\\
\vdots \\
q_n(t)
\end{bmatrix}
\end{equation}
with $P_1$ a nonsingular matrix. By Corollary \ref{cor1}, the degrees of $q_1(t),q_2(t),\ldots,q_n(t)$ are $n-1,n-2,\ldots,0$, and therefore 
\begin{equation}\label{eq2}
\begin{bmatrix}
q_1(t) \\
q_2(t)\\
\vdots \\
q_n(t)
\end{bmatrix}=
P_2\cdot  
\begin{bmatrix}
1 \\
t\\
\vdots \\
t^{n-1}
\end{bmatrix}
\end{equation}
with $P_2$ a nonsingular matrix. By putting together Eq. \eqref{eq1} and Eq. \eqref{eq2}, we get Eq. \eqref{eq-th}, with $A=P_1\cdot P_2$. 

$(\Leftarrow)$ One can directly verify that $W(1,t,\ldots,t^{n-1})=2!\cdot 3!\cdots (n-1)!\neq 0$. If Eq. \eqref{eq-th} holds, then (see Lemma 2 in \cite{BO}) we have that $W(p_1,\ldots,p_n)=|A|\cdot W(1,t,\ldots,t^{n-1})=|A|\cdot 2!\cdot 3!\cdots (n-1)!$. Since $A$ is nonsingular, $|A|\neq 0$ and $W(p_1,\ldots,p_n)$ is a nonzero constant. 
\end{proof}

\subsection{Geometric consequences}\label{geom}

Let us translate Theorem \ref{th-main} into Geometry. In order to do this, let $\bfy(t)=(p_1(t),p_2(t),\ldots,p_n(t))$ define a (polynomial) parametrization of a curve ${\mathcal C}\subset {\Bbb K}^n$, where $p_1(t),\ldots,p_n(t)$ satisfy Theorem \ref{th-main}. Then Theorem \ref{th-main} is saying that ${\mathcal C}$ is the affine image of the curve parametrized by $\bfv(t)=(1,t,\ldots,t^{n-1})$, which is obviously contained in the hyperplane $x_1=1$. Since affine mappings preserve the degree of parametrizations, and map hyperplanes onto hyperplanes, we get a first geometric consequence of Theorem \ref{th-main}.

\begin{theorem}\label{before}
Let $\bfy(t)=(p_1(t),p_2(t),\ldots,p_n(t))$ be a polynomial parametrization of ${\mathcal C}\subset {\Bbb K}^n$ with $W(p_1,...,p_n) \neq 0$ and let ${\bf d}$ be the degree of $\bfy(t)$. Then $W(p_1,\ldots,p_n)$ is a nonzero constant if and only if ${\bf d}=n-1$. Furthermore, in that case ${\mathcal C}$ is contained in a hyperplane. 
\end{theorem}

\begin{corollary}\label{now}
Let $\bfy(t)=(p_1(t),p_2(t),\ldots,p_n(t))$ be a polynomial parametrization of ${\mathcal C}\subset {\Bbb K}^n$. If $W(p_1,\ldots,p_n)$ is a constant (possibly zero) then ${\mathcal C}$ is contained in a hyperplane. 
\end{corollary}

\begin{proof}
If $W(p_1,\ldots,p_n)$ is identically zero then $p_1(t),\ldots,p_n(t)$ are linearly dependent, i.e. there exists $\alpha_1,\ldots,\alpha_n\in {\Bbb K}$, not all of them zero, such that $\alpha_1p_1(t)+\cdots +\alpha_np_n(t)=0$. Therefore, ${\mathcal C}$ is contained in the plane $\alpha_1x_1+\cdots +\alpha_nx_n=0$. If $W(p_1,\ldots,p_n)$ is a nonzero constant, it follows from Theorem \ref{before}.
\end{proof}

\begin{remark}
Notice that the converse of Corollary \ref{now} is not true. For instance, the curve $\bfy(t)=(t^3,t^3+t^2,t^2-2)$ is contained in the plane $x-y+z+2=0$; however, $W(t^3,t^3+t^2,t^2-2)=12t^2$. 
\end{remark}

Additionally, we can give two more geometric consequences of Theorem \ref{th-main} by recalling the \emph{rational normal curve} in dimension $n$, namely the curve in ${\Bbb K}^n$ parametrized by $\bfu(t)=(t,t^2,\ldots,t^n)$. This curve is well-known in Algebraic Geometry, with results going back to Castelnuovo \cite{C86} and Veronese \cite{V82} (see also the Introduction to \cite{CC07}), and can be found in many basic textbooks of Algebraic Geometry as a canonical example of a curve which is not a complete intersection for $n\geq 3$. If we compute the \emph{hodograph}, i.e. the first derivative, of $\bfu(t)=(t,t^2,\ldots,t^n)$, we get $\bfv(t)=\bfu'(t)=(1,2t,\ldots,nt^{n-1})$, so
\[
\bfu'(t)=\begin{bmatrix}
1 & 0 & \cdots & 0\\
0 & 2 & \cdots & 0\\
0 & 0 & \cdots & n
\end{bmatrix}
\begin{bmatrix}
1\\
t\\
\vdots \\
t^{n-1}
\end{bmatrix}.
\]

Thus, we get the following result, which follows directly from Theorem \ref{th-main}.

\begin{theorem} \label{th2}
The rational normal curve is the only polynomial curve, up to affine isomorphism, whose hodograph has a constant nonzero Wronskian.
\end{theorem}

Finally, we can write $W(p_1,\ldots,p_n)$ in terms of $\bfy(t)$ as 
\begin{equation}\label{wrons2}
W(p_1,\ldots,p_n)=\mbox{det}(\bfy(t),\bfy'(t),\ldots,\bfy^{(n-1)}(t)),
\end{equation}
i.e. the determinant whose first row consists of the components of $\bfy(t)$, and whose $j$-th row, $j\geq 1$, consists of the components of the $(j-1)$-th derivative of $\bfy(t)$. We will refer to the right-hand of Eq. \eqref{wrons2} as the \emph{Wronskian of $\bfy(t)$}. As mentioned in the Introduction to the paper, replacing $\bfy(t):=\bfx'(t)$ in Eq. \eqref{wrons2} yields the determinant $\mbox{det}(\bfx'(t),\bfx''(t))$ and $\mbox{det}(\bfx'(t),\bfx''(t),\bfx'''(t))$, appearing as the numerator of the formulas for the \emph{curvature}, in ${\Bbb K}^2$, and the \emph{torsion}, in ${\Bbb K}^3$. Then we get the last consequence of Theorem \ref{th-main}. 

\begin{theorem} \label{th3}
Let ${\mathcal C}\subset {\Bbb R}^n$, with $n=2$ or $n=3$, be polynomially parametrized by $\bfx(t)$. 
\begin{itemize}
\item [(i)] If ${\mathcal C}\subset {\Bbb R}^2$ and $\bfx(t)$ does not have any (real or complex) points with vanishing curvature, then $\bfx(t)$ is an affine image of the rational normal curve in ${\Bbb R}^2$.
\item [(ii)]If ${\mathcal C}\subset {\Bbb R}^3$ and $\bfx(t)$ does not have any (real or complex) points with vanishing torsion, then $\bfx(t)$ is an affine image of the rational normal curve in ${\Bbb R}^3$. 
\end{itemize}
\end{theorem}

\begin{remark}
The value of the Wronskian is related to the parametrization $\bfy(t)$ that we consider, i.e. if we consider a reparametrization $\tilde{\bfy}(t)=(\bfy\circ \phi)(t)$, it is not true in general that $W(\tilde{\bfy}(t))$ is the result of composing $\phi(t)$ with $W(\bfy(t))$. For instance, the Wronskian of $\bfy(t)=(t,t^2-t,t^2+1)$ is $W(\bfy(t))=-2$; however, for $\tilde{\bfy}(t)=(\bfy\circ \phi)(t)=(t^2,t^4-t^2,t^4+1)$, where $\phi(t)=t^2$, we have that $W(\tilde{\bfy}(t))=16t^3$. If $\phi(t)=t+\alpha$, though, then it is definitely true that $W(\tilde{\bfy}(t))=W(\bfy(t))$.  
\end{remark}

\section{Generalization to Laurent Polynomials} \label{laurent}

A \emph{Laurent polynomial} is an expression
\begin{equation}\label{laupol}
p(t)=a_\ell t^\ell+\cdots +a_1 t+ a_0+a_{-1}t^{-1}+\cdots +a_{-k} t^{-k},
\end{equation}
where $\ell,k\in {\Bbb N}\cup \{0\}$. Thus, we are just considering polynomial expressions plus finitely many monomial terms with negative exponents. Laurent polynomials with coefficients in ${\Bbb K}$ form a ring, that we denote by ${\Bbb K}[t,t^{-1}]$. 

\begin{example}
Laurent polynomials with nonzero constant wronskian certainly exist. For instance, 
\[
W\left(t^6,\dfrac{1}{t},\dfrac{1}{t^3}\right)=
\left\vert \begin{array}{ccc}
t^6 & 1/t & 1/t^2\\
6t^5 & -1/t^2 & -2/t^3\\
30 t^4 & 2/t^3 & 6/t^4
\end{array}
\right\vert = -56.
\]
\end{example}

Let us see whether the results of Section \ref{back} are also applicable here. First, it will be more convenient to use, instead of Eq. \eqref{laupol}, the following notation for the Laurent polynomials $p_1(t),\ldots,p_n(t)$: 
\begin{equation}\label{laupol2}
p_i(t)=\sum_{j=M_i}^{N_i} a_j t^{d_j},
\end{equation}
with $M_i\leq N_i$, $M_i,N_i\in {\Bbb Z}$, and at least one $M_k$ satisfies that $M_k<0$ (otherwise we just have polynomials). We will refer to $M_i$ as the {\it minimum} degree of $p_i(t)$, which corresponds to the term of least degree of $p_i(t)$, and we will refer to $N_i$ as the {\it maximum} degree of $p_i(t)$, which corresponds to the term of highest degree of $p_i(t)$. Let 
\[
{\bf m}=\mbox{min}(M_1,\ldots,M_\ell),\mbox{ }{\bf n}=\mbox{max}(N_1,\ldots,N_\ell).
\]
We will refer to ${\bf m}$ (resp. ${\bf n}$) as the {\it minimum} (resp. {\it maximum}) {\it degree of the set} $\{p_1(t),\ldots,p_n(t)\}$. Then $p_1(t),\ldots,p_n(t)$ span a subspace ${\mathcal V}=\langle p_1(t),\ldots,p_n(t)\rangle$ of the vector space ${\Bbb K}_{\bf L}[t,t^{-1}]$, of dimension ${\bf L}={\bf n}-{\bf m}+1$, generated by the monomials $t^r$, where ${\bf m}\leq r \leq {\bf n}$; notice that at least one of these monomials has a negative exponent, i.e. ${\bf m}<0$. Now if we have $p_1(t),\ldots,p_n(t)$ Laurent polynomials such that $W(p_1,\ldots,p_n)$ is a nonzero constant, in particular $p_1,\ldots,p_n$ must be linearly independent as vectors of ${\Bbb K}_{\bf L}[t,t^{-1}]$, so $n\leq {\bf L}$. Additionally, Proposition \ref{prop1} also holds here; the proof works exactly the same, but now we need to take care of both the minimum and the maximum degrees of the polynomials. Thus, we get the following result. 

\begin{proposition}\label{prop2}
Let $p_1(t),\ldots,p_n(t)\in {\Bbb K}[t,t^{-1}]$, linearly independent, and let ${\mathcal V}=\langle p_1(t),\ldots,p_n(t)\rangle$. Then there exist $q_1(t),\ldots,q_n(t)\in {\Bbb K}[t,t^{-1}]$ (resp. $\tilde{q}_1,\ldots,\tilde{q}_n(t)\in {\Bbb K}[t,t^{-1}]$) satisfying that: (i) the maximum (resp. the minimum) degree of $\{p_1,\ldots,p_n\}$ and of $\{q_1,\ldots,q_n\}$ (resp. $\{\tilde{q}_1,\ldots,\tilde{q}_n\}$) is the same; (ii) the maximum degrees of the $q_i(t)$ (resp. the minimum degrees of the $\tilde{q}_i(t)$) are all distinct, i.e. for $i\neq j$ the maximum (resp. minimum) degree of $q_i(t),q_j(t)$ (resp. $\tilde{q}_i(t),\tilde{q}_j(t)$) is different; (iii) $W(p_1,\ldots,p_n)=W(q_1,\ldots,q_n)$ (resp. $W(p_1,\ldots,p_n)=W(\tilde{q}_1,\ldots,\tilde{q}_n)$); (iv) $\langle p_1(t),\ldots,p_n(t)\rangle=\langle q_1(t),\ldots,q_n(t)\rangle$ (resp. $\langle p_1(t),\ldots,p_n(t)\rangle=\langle \tilde{q}_1(t),\ldots,\tilde{q}_n(t)\rangle$). 
\end{proposition}

\begin{example}\label{exlau}
To illustrate Proposition \ref{prop2}, consider $p_1(t)=t^2-1+t^{-1}$, $p_2(t)=t^2+t-t^{-2}$, $p_3(t)=t+t^{-2}$. Then
\[
W(p_1,p_2,p_3)=\left\vert \begin{array}{ccc}
t^2-1+t^{-1} & t^2+t-t^{-2} & t+t^{-2}\\
2t-t^{-2} & 2t+1+2t^{-2} & 1-2t^{-3}\\
2+2t^{-3} & 2-6t^{-4} & 6t^{-4}
\end{array}
\right \vert=2 - 3t^{-1} - 12t^{-2} - 8t^{-3} + 6t^{-5}.
\]
Notice that the maximum degrees of $p_1(t),p_2(t),p_3(t)$ are $2,2,1$. Denoting the $j$-th column of $W(p_1,p_2,p_3)$ by ${\bf c}_j$, we perform the operations ${\bf c}_2:={\bf c}_2-{\bf c}_1$, ${\bf c}_3:={\bf c}_3-{\bf c}_2$, and we get that 
\[
W(p_1,p_2,p_3)=W(t^2-1+t^{-1},t+1-t^{-1}-t^{-2},-1+t^{-1}+2t^{-2}),
\]
i.e. $W(p_1,p_2,p_3)=W(q_1,q_2,q_3)$ with $q_1(t)=p_1(t)=t^2-1+t^{-1}$, $q_2(t)=t+1-t^{-1}-t^{-2}$, $q_3(t)=-1+t^{-1}+2t^{-2}$; the maximum degrees of the $q_i(t)$ are now $2,1,0$, so all of them are distinct. Additionally, notice that the minimum degrees of $p_1(t),p_2(t),p_3(t)$ are $-1,-2,-2$. If we perform on $W(p_1,p_2,p_3)$ the operation ${\bf c}_2:={\bf c}_2+{\bf c}_3$, we get that (now we order the terms of each $p_i(t)$ by increasing degree)
\[
W(p_1,p_2,p_3)=W(t^{-1}-1+t^2,-t^{-2}+t+t^2,t^{-2}+t)=W(t^{-1}-1+t^2,2t+t^2,t^{-2}+t),
\]
i.e. $W(p_1,p_2,p_3)=W(\tilde{q}_1,\tilde{q}_2,\tilde{q}_3)$ with $\tilde{q}_1(t)=p_1(t)=t^{-1}-1+t^2$, $\tilde{q}_2(t)=2t+t^2$, $\tilde{q}_3(t)=p_3(t)=t^{-2}+t$, whose minimum degrees are $-1,1,-2$; notice that all of them are distinct. 
\end{example}

The Wronskian $W(p_1,\ldots,p_n)$ of a finite set of Laurent polynomials is also a Laurent polynomial, so it makes sense to speak about the minimum and the maximum degree of $W(p_1,\ldots,p_n)$. Then we have the following result, which is analogous to Proposition \ref{propdeg} for the case of Laurent polynomials. 

\begin{proposition}\label{propdeg2}
Let $q_1(t),\ldots,q_n(t)\in {\Bbb K}[t,t^{-1}]$, and for $i=1,\ldots,n$ let $d_i$ be the maximum degree (resp. $e_i$ the minimum degree) of $q_i(t)$. Assume that all the $d_i$ (resp. the $e_i$) are distinct. Then the maximum degree $\tau_{max}$ (resp. the minimum degree $\tau_{min}$) of $W(q_1,\ldots,q_n)$ is 
\begin{equation}\label{deg2}
\tau_{max}=d_1+\cdots+d_n-{n\choose 2},\mbox{ }\tau_{min}=e_1+\cdots+e_n-{n\choose 2}.
\end{equation}
\end{proposition}

\begin{proof}
By writing the Wronskian $W(q_1,\ldots,q_n)$ as a sum of determinants, the term of maximum (resp. minimum) degree corresponds to the Wronskian of the terms of maximum (resp. minimum) degree of the $q_i(t)$. Then we can proceed as in the proof of Proposition \ref{propdeg}.
\end{proof}

In order to characterize the Laurent polynomials with nonzero constant Wronskian, we need first the following technical lemma. 

\begin{lemma}\label{now2}
Let $p_1(t),\ldots,p_n(t)\in {\Bbb K}[t,t^{-1}]$, linearly independent. Then we can find $q_1(t),\ldots,q_n(t)$ and $\tilde{q}_1(t),\ldots,\tilde{q}_n(t)$, with the properties in Proposition \ref{prop2}, such that the maximum degrees $d_i$ of the $q_i(t)$ and the minimum degrees $e_i$ of the $\tilde{q}_i(t)$ satisfy that $d_i\geq e_i$. 
\end{lemma}

\begin{proof}
Let $q_1(t),\ldots,q_n(t)$ be computed as in Proposition \ref{prop2}, and let $d_i$ be the maximum degree of $q_i(t)$. Then we can compute the $\tilde{q}_i(t)$ starting directly from the $q_i(t)$ (and not from the $p_i(t)$), and the result follows. 
\end{proof}

In the case of nonzero constant Wronskian, we can push the result a bit further. 

\begin{proposition}\label{essential}
Let $p_1(t),\ldots,p_n(t)\in {\Bbb K}[t,t^{-1}]$, linearly independent, and assume that $W(p_1,\ldots,p_n)$ is a nonzero constant. Let $q_1(t),\ldots,q_n(t)$ and $\tilde{q}_1(t),\ldots,\tilde{q}_n(t)$ be as in Lemma \ref{now2}, and for $i=1,\ldots,n$, let $d_i$ (resp. $e_i$) be the maximum (resp. the minimum) degree of $q_i(t)$ (resp. $\tilde{q}_i(t)$). Then we have that:
\begin{itemize}
\item [(1)] $d_1+\cdots+d_n=e_1+\cdots+e_n=\displaystyle{n\choose 2}$.
\item [(2)] $d_i=e_i$ for $i=1,\ldots,n$.
\end{itemize}
\end{proposition}

\begin{proof}
By Proposition \ref{prop2}, $W(p_1,\ldots,p_n)=W(q_1,\ldots,q_n)=W(\tilde{q}_1,\ldots,\tilde{q}_n)$, so $W(q_1,\ldots,q_n)$ and $W(\tilde{q}_1,\ldots,\tilde{q}_n)$ are nonzero constants. Therefore, by Proposition \ref{propdeg2} we have that $\tau_{max}=\tau_{min}=0$, and statement (1) follows. From statement (1) we get that $(d_1-e_1)+\cdots +(d_n-e_n)=0$. Since $d_i-e_i\geq 0$ by Lemma \ref{now2}, we deduce that $d_i-e_i=0$ for all $i=1,\ldots,n$, and (2) follows. 
\end{proof}

In turn, Proposition \ref{essential} leads to the following result. 

\begin{proposition}\label{essential2}
Let $p_1(t),\ldots,p_n(t)\in {\Bbb K}[t,t^{-1}]$, linearly independent, and assume that $W(p_1,\ldots,p_n)$ is a nonzero constant. For $i=1,\ldots,n$, let $q_i(t),\tilde{q}_i(t)$ be as in Lemma \ref{now2}, and let $d_i$ be the maximum (resp. minimum) degree of $q_i(t)$ (resp. $\tilde{q}_i(t))$. Then 
\[
\{t^{d_1},\ldots,t^{d_n}\}
\]
is a basis for $\langle p_1(t),\ldots,p_n(t)\rangle$.
\end{proposition}

\begin{proof}
Because of Proposition \ref{prop2} we know that $\langle p_1(t),\ldots,p_n(t)\rangle=\langle q_1(t),\ldots,q_n(t)\rangle$ and $\langle p_1(t),\ldots,p_n(t)\rangle=\langle \tilde{q}_1(t),\ldots,\tilde{q}_n(t)\rangle$, so $\langle q_1(t),\ldots,q_n(t)\rangle=\langle \tilde{q}_1(t),\ldots,\tilde{q}_n(t)\rangle$. Assume without loss of generality that $d_1>d_2>\cdots>d_n$, so $d_1$ is the maximum of the $d_i$ (which are all of them distinct). Then $q_1(t)$ consists of monomials of degree $\leq d_1$. Additionally, since by Proposition \ref{essential} $d_1$ is also the minimum degree of $\tilde{q}_1(t)$, we have that $\tilde{q}_1(t)$ consists of monomials of degree $\geq d_1$. Now if $\tilde{q}_1(t)$ contains some monomial $t^{d'_1}$ of degree $d'_1>d_1$, we get that $t^{d'_1}$ must be part of a monomial basis for $\langle \tilde{q}_1(t),\ldots,\tilde{q}_n(t)\rangle$. Since $\langle q_1(t),\ldots,q_n(t)\rangle=\langle \tilde{q}_1(t),\ldots,\tilde{q}_n(t)\rangle$, we deduce that $t^{d'_1}$ is also part of a monomial basis for $\langle q_1(t),\ldots,q_n(t)\rangle$, which does not make sense because $d_1$ is the maximum degree appearing in the $q_i(t)$, so no term in $t^{d'_1}$, $d'_1>d_1$, is present in any $q_i(t)$. Thus, $\tilde{q}_1(t)=\alpha_1 t^{d_1}$, $\alpha_1\neq 0$. 

Next, consider $\tilde{q}_2(t)$. We know that the minimum degree of $\tilde{q}_2(t)$ is $d_2$, i.e. $\tilde{q}_2(t)$ consists of monomials of degree $\geq d_2$. However, arguing as before, we deduce that $\tilde{q}_2(t)$ can only have monomials in $t^{d_2}$ and $t^{d_1}$. Proceeding in the same way for $\tilde{q}_i(t)$ with $i\geq 3$, we deduce that $\{t^{d_1},\ldots,t^{d_n}\}$ is a basis for $\langle \tilde{q}_1(t),\ldots,\tilde{q}_n(t)\rangle$. Since $\langle p_1(t),\ldots,p_n(t)\rangle=\langle \tilde{q}_1(t),\ldots,\tilde{q}_n(t)\rangle$, the result follows. 
\end{proof}

Finally we can characterize what we want.

\begin{theorem}\label{th-main2}
Let $p_1(t),\ldots,p_n(t)\in {\Bbb K}[t,t^{-1}]$, linearly independent. Then $W(p_1,\ldots,p_n)$ is a nonzero constant if and only if we have that 
\begin{equation}\label{eq-th2}
\begin{bmatrix}
p_1(t) \\
p_2(t)\\
\vdots \\
p_n(t)
\end{bmatrix}=A\cdot  
\begin{bmatrix}
t^{r_1} \\
t^{r_2}\\
\vdots \\
t^{r_n}
\end{bmatrix}
\end{equation}
where $A$ is a nonsingular matrix, and $r_1,\ldots,r_n\in {\Bbb Z}$, all distinct, and satisfy that 
\begin{equation}\label{lacondi}
r_1+\cdots+r_n=\displaystyle{n\choose 2}.
\end{equation}
\end{theorem}

\begin{proof}
$(\Rightarrow)$ follows from Proposition \ref{essential2} and Proposition \ref{essential}. For $(\Leftarrow)$, notice that $W(t^{r_1},\ldots,t^{r_n})$ is equal to 
\[
\Pi_{i=1}^n r_i(r_i-1)\cdots (r_i-n+1)
\]
multiplied by a Vandermonde determinant, which is nonzero because all the $r_i$ are distinct. Using Lemma 2 in \cite{BO} as in the proof of the implication $(\Leftarrow)$ of Theorem \ref{th-main}, we are done. 
\end{proof}

Notice that compared to the case in Section \ref{back}, now we have much more freedom. For $n=2$, Theorem \ref{th-main2} implies that $r_1+r_2=1$, with $r_1\neq r_2$, $r_1,r_2\geq {\Bbb Z}$, so we have infinitely many pairs fulfilling these conditions. For $n\geq 3$ we have even more freedom. 

\begin{example}
The Laurent polynomials $p_1(t)=t^6$, $p_2(t)=1/t$, $p_3(t)=1/t^2$ of Example \ref{exlau} have degrees $6,-1,-2$, so 
\[
6+(-1)+(-2)=3=\displaystyle{3\choose 2},
\]
and Eq. \eqref{lacondi} is fulfilled. 
\end{example}

\para
\begin{remark}
Let $\bfy(t)=(p_1(t),\ldots,p_n(t))$ be a parametrization of a curve ${\mathcal C}\subset {\Bbb K}^n$, where the $p_i(t)$ are Laurent polynomials. The parametrizations $(t^{r_1},\ldots,t^{r_n})$ with $r_i\in {\Bbb Z}$ are examples of curves which are \emph{toric varieties} \cite{CLS} in dimension $n$. Thus, Eq. \eqref{lacondi} is the condition for a toric curve in dimension $n$, with distinct $r_i$, to have a nonzero constant Wronskian. In particular, Theorem \ref{th-main2} says that the parametric curves, whose components are Laurent polynomials, with nonzero constant Wronskian are exactly the affine images of such curves. Furthermore, notice that to achieve Eq. \eqref{lacondi} we need to have both positive and negative $r_i$, so these curves are not bounded, i.e. they have a branch at infinity. Additionally, unlike Section \ref{back}, the degree of these curves is not bounded, since Eq. \eqref{lacondi} can be achieved with one of the $r_i$ arbitrarily large. Also unlike Section \ref{back}, these curves are not necessarily contained in a hyperplane. 
\end{remark}

\begin{remark}
The results in this section are also valid for polynomials $\tilde{p}(t)=(p\circ \phi)(t)$ where $\phi(t)=t+a$ and $p(t)$ is a Laurent polynomial, because $W(\tilde{p}_1,\ldots,\tilde{p}_n)=W(p_1,\ldots,p_n)$. In other words, the results are valid whenever the set $\{p_1(t),\ldots,p_n(t)\}$ consists of rational functions with only one pole. 
\end{remark}

\section{Rational curves with more than one pole}\label{sec-impossible}

Finally we consider a set of rational functions $p_1(t),\ldots,p_n(t)\in {\Bbb K}(t)$, linearly independent, with at least two different poles (since otherwise we are in the cases treated in the previous sections). We want to investigate whether it is possible that the Wronskian $W(p_1,\ldots,p_n)$ is a nonzero constant. Since we can always apply a linear reparametrization $t\to t+a$ taking one of these poles to the origin (which does not change the value of the Wronskian), we will assume that the poles of the $p_i(t)$ are $0,\beta_1,\ldots,\beta_r$ where $r\geq 1$, the $\beta_j$ are distinct, and $\beta_j\neq 0$ for all $j=1,\ldots,r$. Therefore, every $p_i(t)$ looks like
\begin{equation}\label{theexp}
\begin{array}{rcl}
p_i(t)&=&a_{-k}t^{-k}+\cdots +a_{-1}t^{-1}+a_0+a_1 t+\cdots +a_s t^s+\\
&&+b_{-\ell_1}(t-\beta_1)^{-\ell_1}+\cdots +b_{-1}(t-\beta_1)^{-1}+\\
&&+c_{-\ell_r}(t-\beta_r)^{-\ell_k}+\cdots +c_{-1}(t-\beta_r)^{-1},
\end{array}
\end{equation}
where of course the coefficients and exponents vary with $i$ (although we do not spell this out to avoid a cumbersome notation). Furthermore, among the $p_i(t)$ we can find polynomials with some nonzero $a_{-j}$, where $j\in {\Bbb N}$, and analogously for $b_{-m},\ldots,c_{-n}$, with $m,\ldots,n\in {\Bbb N}$. Denoting by $K,L_1$ the maximum orders of the poles $t=0$, $t=\beta_1$ among the $p_i(t)$, we can always assume that nonzero terms in $t^{-K}$ and $(t-\beta_1)^{-L_1}$ are present in all the $p_i$: indeed, writing $\bfy(t)=(p_1(t),\ldots,p_n(t))$, if $A$ is an orthogonal $n\times n$ matrix, we know that $W(A\bfy(t))=|A|W(\bfy(t))$, so if $W(\bfy(t))$ is a nonzero constant, so is $W(A\bfy(t))$; but by appropriately choosing $A$, we can be sure that all the components of $A\bfy(t)$ have nonzero terms in $t^{-K}$ and $(t-\beta_1)^{-L_1}$. Additionally, proceeding with Gauss-like elimination as we did in Proposition \ref{prop1} and Proposition \ref{prop2}, we can equivalently consider $W(q_1,\ldots,q_n)=W(p_1,\ldots,p_n)$ (possibly up to a change in the sign) where $q_1$ contains a nonzero term in $t^{-K}$, $q_2$ does not contain any term in $t^{-K}$, but contains a nonzero term in $(t-\beta_1)^L_1$, and $q_i$, for $i\geq 3$, does not contain nonzero terms in $t^{-K}$, $(t-\beta_1)^L_1$. 

\para
From now on we will focus on the case $n=2$. Here we have the following impossibility result. 

\begin{theorem}\label{notpossibledim2}
Let $p_1(t),p_2(t)\in {\Bbb K}(t)$, linearly independent, where the $p_i(t)$ have at least two different poles. Then the Wronskian $W(p_1,p_2)$ cannot be a nonzero constant. 
\end{theorem}

\begin{proof}
Since $n=2$ we just have two rational functions $p_1(t),p_2(t)$; from previous considerations, we can assume that the first one contains a nonzero term in $t^{-K}$, and that the second does not contain such a term, but contains a nonzero term in $(t-\beta_1)^{-L_1}$. Next if we form the Wronskian $W(p_1,p_2)$, and decompose it into a sum of determinants after writing each $p_i(t)$ as in Eq. \eqref{theexp}, one of these determinants must be $W(t^{-K},(t-\beta_1)^{-L_1})$; in more detail, 
\[
W(t^{-K},(t-\beta_1)^{-L_1})=\left\vert \begin{array}{cc}
t^{-K} & (t-\beta_1)^{-L_1}\\
-Kt^{-K-1} & -L_1 (t-\beta_1)^{-L_1-1}
\end{array}\right\vert =t^{-K-1}(t-\beta_1)^{\-L_1-1} f(t),
\]
where $f(t)=(-L_1+K)t-K\beta_1$. Since $t=0$ is a pole $K\neq 0$, and since $t=\beta_1$ and $t=0$ are distinct poles, $\beta_1\neq 0$; also, since $t=\beta_1$ is a pole $L_1\neq 0$. Thus, $f(0)=-K\beta_1\neq 0$, and $f(\beta_1)=-L_1\beta_1\neq 0$, so $W(t^{-K},(t-\beta_1)^{-L_1})$ provides a rational term of $W(p_1,p_2)$ which has $t=0$ as a pole of order $K+1$, and $t=\beta_1$ as a pole or order $L_1+1$. However, any other determinant appearing when writing $W(p_1,p_2)$ as a sum of determinants yields a rational function where either $t=0$ is not a pole, or $t=\beta_1$ is not a pole, or $t=0,t=\beta_1$ are poles but with orders different from $K+1,L_1+1$, since at least one of them should be decreased. Therefore, it is impossible that $W(p_1,p_2)$ is a constant.
\end{proof}

Unfortunately, and even if it sounds plausible, generalizing the above argument for $n\geq 3$ is difficult. Still, we could not find any example of $n\geq 3$ rational functions with at least two poles with nonzero constant Wronskian, and we conjecture that Theorem \ref{notpossibledim2} is also true for $n\geq 3$. But proving or disproving this is left here as an open problem.

\section{Conclusion and Further Work}\label{sec-conclusion}

Motivated by some questions in Differential Geometry, more specifically by problems regarding symmetry detection and detecting affine and projective equivalences between rational curves, we have investigated whether a set of linearly independent polynomials, Laurent polynomials and rational functions can have a nonzero constant Wronskian. The answer was positive in the first two cases, and we could characterize the cases where nonzero constant Wronskians appear, and extract some geometric consequences, especially for polynomials. Additionally, for rational cases and whenever $n=2$, the answer is negative; for $n\geq 3$ we conjecture that the answer is also negative, but this is left here as an open question. 

\para
The natural continuation of the problem would be to move to a bigger realm of functions with good properties, e.g. analytic functions. Certainly, the phenomenon can also happen in that case: if $f_1(t)=\mbox{cos}(t)$, $f_2(t)=\mbox{sin}(t)$, $W(f_1,f_1)=1$. However, in that case we need to work with {\it formal power series} instead of expressions with finitely many terms, as we did in this paper. So this is certainly more challenging, and we also pose it here as an open problem.

\end{document}